\documentclass[11pt,english,british]{article}
\usepackage[T1]{fontenc}
\usepackage[latin9]{inputenc}
\usepackage{geometry}
\usepackage{amsmath, amsfonts, amsthm, amssymb}
\usepackage{graphicx}
\usepackage{float}
\usepackage{verbatim}

\numberwithin{equation}{section}
\newtheorem{thm}{Theorem}[section]
\newtheorem{lma}[thm]{Lemma}
\newtheorem{cor}[thm]{Corollary}

\newtheorem{prop}[thm]{Proposition}

\newtheorem{ques}[thm]{Question}

\renewcommand{\geq}{\geqslant}
\renewcommand{\leq}{\leqslant}
\renewcommand{\H}{\text{H}}
\renewcommand{\P}{\text{P}}

\usepackage{babel}
  \addto\captionsbritish{}
  \addto\captionsbritish{}
  \addto\captionsbritish{}
  \addto\captionsenglish{}
  \addto\captionsenglish{}
  \addto\captionsenglish{}

\title{On the equality of Hausdorff measure and Hausdorff content}

\author{\'Abel Farkas$^{1}$ and Jonathan M. Fraser$^{2}$}

\begin{document}
      \maketitle

\begin{center}
$^1$Mathematical Institute, University of St Andrews, North Haugh, St Andrews, Fife, KY16 9SS, UK.\\
Email: af56@st-andrews.ac.uk
\\ \vspace{3mm} 
  $^2$School of Mathematics, The University of Manchester, Manchester, M13 9PL, UK.\\
  Email: jonathan.fraser@manchester.ac.uk
\end{center}
\vspace{0mm}

\begin{abstract}
We are interested in situations where the Hausdorff measure and Hausdorff content of a set are equal in the critical dimension.  Our main result shows that this equality holds for any subset of a self-similar set corresponding to a nontrivial cylinder of an irreducible subshift of finite type, and thus also for any self-similar or graph-directed self-similar set, regardless of separation conditions.  The main tool in the proof is an exhaustion lemma for Hausdorff measure based on the Vitali Covering Theorem.

We also give several examples showing that one cannot hope for the equality to hold in general if one moves in a number of the natural directions away from `self-similar'.  For example, it fails in general for self-conformal sets, self-affine sets and Julia sets.  We also give applications of our results concerning Ahlfors regularity.  Finally we consider an analogous version of the problem for packing measure.  In this case we need the strong separation condition and can only prove that the packing measure and $\delta$-approximate packing pre-measure coincide for sufficiently small $\delta>0$.
\\ \\
\noindent \emph{AMS Classification 2010:} primary: 28A78, secondary: 28A80, 37C45. \\
\noindent \emph{Keywords}: Hausdorff measure, Hausdorff content, packing measure, self-similar set, subshift of finite type.

\end{abstract}

\section{Introduction}

Hausdorff measure and dimension are among the most important notions in fractal geometry and geometric measure theory used to quantify the size of a set.  The Hausdorff content is a concept closely related to the Hausdorff measure, but perhaps less popular in the context of classical measure theory.  That being said the Hausdorff content enjoys greater regularity than the Hausdorff measure and still gives the Hausdorff dimension as the critical exponent.  The goal of this article is to understand further the relationship between Hausdorff measure and Hausdorff content in the context of some well-known and popular classes of fractals sets.  In particular we are interested in when the Hausdorff measure and Hausdorff content of a set are equal in the Hausdorff dimension.  This study was motivated by a question of Michael Barnsley posed to one of the authors.

\subsection{Hausdorff measure and Hausdorff content}

Let $F \subseteq \mathbb{R}^n$.  For $s \geq 0$ and $\delta>0$ the \emph{$\delta$-approximate $s$-dimensional Hausdorff measure} of $F$ is defined by
\begin{eqnarray*}
\mathcal{H}_\delta^s (F) = \inf \bigg\{ \sum_{k=1}^{\infty} \text{diam}( U_k)^s : \{ U_k\}_{k=1}^{\infty} \text{ is a countable cover of $F$ }\\ 
  \text{   by sets with $\text{diam}(U_k) \leq \delta$ for all $k$} \bigg\}
\end{eqnarray*}
and the $s$-dimensional Hausdorff (outer) measure of $F$ by $\mathcal{H}^s (F) = \lim_{\delta \to 0} \mathcal{H}_\delta^s (F)$.  If one does not put any restriction on the diameters of the covering sets, then one obtains the Hausdorff content of $F$, namely,
\[
\mathcal{H}_\infty^s (F) = \inf \bigg\{ \sum_{k=1}^{\infty} \text{diam}( U_k)^s : \{ U_k\}_{k \in \mathcal{K}} \text{ is a countable cover of $F$ by arbitrary sets} \bigg\}.
\]
The following chain of inequalities is evident
\[
\mathcal{H}_\infty^s (F) \leq \mathcal{H}_\delta^s (F) \leq \mathcal{H}^s (F) 
\]
(for every $\delta>0$) and, moreover, the Hausdorff dimension of $F$ is equal to 
\begin{eqnarray*}
\dim_\text{H} F   \ = \ \inf \Big\{ s \geq 0: \mathcal{H}^s (F) =0 \Big\}  \ = \  \inf \Big\{ s \geq 0: \mathcal{H}^s_\infty (F) =0 \Big\}.
\end{eqnarray*}
Thus, for every $s> \dim_\text{H} F$, we have $ \mathcal{H}^s_\infty (F) =  \mathcal{H}^s_\delta (F) =  \mathcal{H}^s (F) =0$ and for every $s< \dim_\text{H} F$, we have $ \mathcal{H}_\infty^s (F) \leq  \mathcal{H}^s_\delta (F) \leq  \mathcal{H}^s (F) = \infty$, again for every $\delta>0$, with the final inequality strict if $F$ is bounded ($\mathcal{H}^s_\delta(F)$ is finite for every $\delta$ for $F$ bounded).  The case when $s =\dim_\text{H} F$ is more subtle, and the case of interest.  Then $\mathcal{H}^s (F)$ may be zero, positive and finite, or infinite, but $\mathcal{H}^s_\infty (F)$ must be finite if $F$ is bounded.  Moreover, if $\mathcal{H}^s_\infty (F) = 0$, then $\mathcal{H}^s (F) = 0$ also.  
\\ \\
The goal of this article is to study situations where $\mathcal{H}^s_\infty(F) = \mathcal{H}^s (F) $ with $s=\dim_\text{H} F$.  Sets with this property were studied by Foran \cite{foran}, where they were called $s$-straight sets, and later studied by Delaware \cite{delaware0, delaware}. There are many advantages to having this equality as Hausdorff content is more easily analysed.  For example, the expression  $\sum_{k=1}^{\infty} \text{diam}( U_k)^s$ gives a genuine upper bound for $\mathcal{H}^s_\infty(F)$ for every cover $\{U_k\}_{k=1}^{\infty}$, and for every $s \geq 0$ the function $\mathcal{H}^s_\infty$ acting on the set of compact subsets of a compact metric space equipped with the Hausdorff metric is an upper semicontinuous function, and thus Baire 1, whereas $\mathcal{H}^s$ is only Baire 2, see \cite{mattilamauldin}.  Another consequence is that $ \mathcal{H}_\delta^s (F) =  \mathcal{H}^s (F)$ for all $\delta>0$.  For more details on Hausdorff measure and dimension, see \cite[Chapter 3]{falconer} and \cite{rogers}.
\\ \\
We conclude this section with a well-known observation and include the proof for completeness.

\begin{lma} \label{haus11}
Let $F \subseteq \mathbb{R}^n$ be such that $\mathcal{H}^s_\infty(F) = \mathcal{H}^s (F)<\infty $ where $s=\dim_\text{\emph{H}} F$.  Then for every $\mathcal{H}^s$-measurable subset $E \subseteq F$ we also have $\mathcal{H}^s_\infty(E) = \mathcal{H}^s (E)$.
\end{lma}

\begin{proof}
A routine calculation using  $\mathcal{H}^s$-measurable hulls verifies that $\mathcal{H}^s(E)  \ = \  \mathcal{H}^s(F)  \, - \,   \mathcal{H}^s(F \setminus E)$ even if $F$ is not $\mathcal{H}^s$-measurable. Therefore
\[
\mathcal{H}^s_\infty(E) \  \leq  \ \mathcal{H}^s(E)  \ = \  \mathcal{H}^s(F)  \, - \,   \mathcal{H}^s(F \setminus E) \  \leq  \ \mathcal{H}^s_\infty(F) \, - \,   \mathcal{H}^s_\infty(F \setminus E)  \ \leq  \ \mathcal{H}^s_\infty(E)
\]
which completes the proof.
\end{proof}

Of course, this result is not necessarily true if we replace $s$ by $\dim_\H E$.

\section{Main results: general situations where $ \mathcal{H}_\infty^s (F) =  \mathcal{H}^s (F)$ }

Let $\mathcal{I} = \{0, \dots, M-1\}$ be a finite alphabet, let $\Sigma  = \mathcal{I}^{\mathbb{N}}$ and $\sigma : \Sigma \to \Sigma$ be the one-sided left shift.  We will write $i \in \mathcal{I}$, $\textbf{\emph{i}} = (i_0 , \dots, i_{k-1}) \in \mathcal{I}^k$ and $\alpha = (\alpha_0, \alpha_1, \dots) \in \Sigma$.  We will also write $\alpha\vert_k = (\alpha_0, \dots, \alpha_{k-1}) \in \mathcal{I}^k$ for the restriction of $\alpha$ to its first $k$ coordinates.  We equip $\Sigma$ with the standard metric defined by
\[
d(\alpha,\beta) = 2^{-n(\alpha,\beta)}
\]
for $\alpha \neq \beta$, where $n(\alpha,\beta) = \max\{ n \in \mathbb{N} : \alpha\vert_n = \beta\vert_n\}$.  We write $\mathcal{I}^* = \cup_{k \in \mathbb{N}} \mathcal{I}^k$ for the set of all finite words. For $\textbf{\emph{i}} = (i_0 , \dots, i_{k-1}) \in \mathcal{I}^*$, we write
\[
[\textbf{\emph{i}}]  = \big\{ \alpha \in \Sigma : \alpha \vert_k = \textbf{\emph{i}} \big\}
\]
for the  \emph{cylinder} corresponding to $\textbf{\emph{i}}$ and we let $|\textbf{\emph{i}}|=k$ be the length of $\textbf{\emph{i}}$.  Also, even though the shift is only defined on $\Sigma$, it will be convenient also to define it for $\textbf{\emph{i}} = (i_0 , \dots, i_{k-1}) \in \mathcal{I}^*$ by
\[
\sigma( \textbf{\emph{i}} ) = \sigma \big( (i_0 , \dots, i_{k-1}) \big) = (i_1 , \dots, i_{k-1}).
\]
Any closed $\sigma$-invariant set $\Lambda \subseteq \Sigma$ is called a \emph{subshift}.  Among the most important subshifts are \emph{subshifts of finite type} which we define as follows.   Let $A$ be an $M \times M$ \emph{transition matrix} indexed by $\mathcal{I} \times \mathcal{I}$ with entries in $\{0,1\}$.  We define the subshift of finite type corresponding to $A$ as
\[
\Sigma_A \ = \ \Big\{\alpha = (\alpha_0 \alpha_1 \dots ) \in \Sigma : A_{\alpha_i, \alpha_{i+1}} = 1 \text{ for all } i =0, 1, \dots  \Big\}.
\]
If every entry of $A$ is 1 then we call $\Sigma_A = \Sigma$ the \emph{full shift}.  We say $\Sigma_A$ is irreducible (or transitive) if the matrix $A$ is irreducible, which means that for all pairs $i,j \in \mathcal{I}$, there exists $n \in \mathbb{N}$ such that $(A^n)_{i,j} >0$.  We say $\Sigma_A$ is aperiodic (or mixing) if the matrix $A$ is aperiodic, which means that there exists $n \in \mathbb{N}$ such that $(A^n)_{i,j} >0$ for all pairs $i,j \in \mathcal{I}$ simultaneously.
\\ \\
To each $i \in \mathcal{I}$ associate a similarity map $S_i$ on $\mathbb{R}^n$ with contraction ratio $r_i \in (0,1)$ which we assume for convenience maps $[0,1]^n$ into itself.   For $\textbf{\emph{i}} = (i_0 , \dots, i_{k-1}) \in \mathcal{I}^*$, write
\[
S_{\textbf{\emph{i}}} = S_{i_0} \circ \cdots  \circ S_{i_{k-1}}
\]
and
\[
r_{\textbf{\emph{i}}} = r_{i_0} \cdots r_{i_{k-1}}.
\]
Let $\Pi: \Sigma \to [0,1]^n$ be the natural coding map given by
\[
\Pi(\alpha) = \bigcap_{k=1}^\infty S_{\alpha\vert_k} \big([0,1]^n\big).
\]
For a given subshift of finite type $\Sigma_A$, we are interested in the set $F_A:=\Pi(\Sigma_A)$.  The set $F:=\Pi(\Sigma)$ corresponding to the full shift is called a \emph{self-similar set} and is the unique non-empty compact set satisfying
\[
F \ = \ \bigcup_{i \in \mathcal{I}} S_i (F).
\]
The collection of contracting similarities $\{S_i\}_{i \in \mathcal{I}}$ is called an \emph{iterated function system} (IFS), see \cite[Chapter 9]{falconer}.  We will also be interested in subsets of $F_A$ corresponding to the cylinders of $\Sigma_A$.  In particular, for $\textbf{\emph{i}} \in \mathcal{I}^*$, let
\[
F_A^\textbf{\emph{i}} = \Pi(\Sigma_A \cap [\textbf{\emph{i}}]),
\]
which may be empty.  It can be shown via the implicit theorems of Falconer \cite{implicit}, \cite[Section 3.1]{techniques} that if $A$ is irreducible, then  $\mathcal{H}^s(F_A) < \infty$ where $s = \dim_\H F_A$.  Moreover, if $\mathcal{H}^s(F_A) > 0$, then $\mathcal{H}^s(F_A^\textbf{\emph{i}}) > 0$ for each $\textbf{\emph{i}} \in \mathcal{I}^*$ for which $F_A^\textbf{\emph{i}} \neq \emptyset$.

\begin{thm} \label{main}
Let $A$ be irreducible and let $s = \dim_\text{\emph{H}} F_A$.  For all $\textbf{i} \in \mathcal{I}^*$ we have
\[
\mathcal{H}_\infty^s \big(F_A^\textbf{i}\big) =  \mathcal{H}^s \big(F_A^\textbf{i}\big).
\]
Moreover, we can extend this to unions of 1-cylinders in the same `family'.  For all $i \in \mathcal{I}$,
\[
\mathcal{H}_\infty^s  \Bigg( \bigcup_{j \in \mathcal{I} : A_{i,j} = 1}  F_A^j \Bigg) \  =  \  \mathcal{H}^s \Bigg( \bigcup_{j \in \mathcal{I} : A_{i,j} = 1}  F_A^j \Bigg).
\]
\end{thm}

We will prove Theorem \ref{main} in Section \ref{mainproof}.  It is natural to wonder if the equality is still satisfied for the full set, and not just cylinders and unions of cylinders in the same family.  We give an example in Section \ref{examples} which shows that this is not true.  Delaware \cite{delaware} proved that any set with finite $\mathcal{H}^s$ measure is $\sigma s$-straight, in that it can be decomposed as a countable union of $s$-straight sets.  This proved a conjecture of Foran \cite{foran}.  Theorem \ref{main} can be viewed as a strengthening of this result in the very special case of subshifts of finite type for self-similar sets.  In particular, we prove that for irreducible $A$ the set $F_A$ can be decomposed into a \emph{finite} union of $s$-straight sets
\[
F_A = \bigcup_{i \in \mathcal{I}} F_A^i.
\]
In this paper we only consider subshifts of finite type, but the same questions are valid for general subshifts and we therefore ask the following natural question.

\begin{ques}
	Does there exists a system of similarities and a transitive subshift $\Lambda \subseteq \Sigma$, such that
	\[
	\mathcal{H}_\infty^s \big(\Pi(\Lambda) \cap [\textbf{i}]\big) <  \mathcal{H}^s \big(\Pi(\Lambda) \cap [\textbf{i}]\big)
	\]
	for some $\textbf{i} \in \mathcal{I}^*$?
\end{ques}
Note that for general subshifts, being \emph{transitive} means that there exists one dense orbit under the left shift.  Transitive subshifts of finite type are precisely those with irreducible $A$ and so Theorem \ref{main} answers this question in the negative for subshifts of finite type.
\\ \\
Theorem \ref{main} was shown for self-similar sets rather than subshifts of finite type by Bandt and Graf \cite[Proposition 3]{bandt-graf} assuming the \emph{open set condition} is satisfied.  See \cite[Section 9.2]{falconer} for the definition and further properties of the open set condition.  This result was generalised by Farkas \cite[Proposition 1.11]{farkas} for self-similar sets without assuming any separation condition. We state this result as a corollary of Theorem \ref{main}.

\begin{cor} \label{corselfsim}
Let $F \subseteq [0,1]^n$ be a self-similar set and let $s = \dim_\text{\emph{H}} F$.  Then, regardless of separation conditions, $\mathcal{H}_\infty^s (F) =  \mathcal{H}^s (F)$.
\end{cor}

\begin{proof}
Since $F$ is self-similar it is modelled by a full shift and thus for any $i \in \mathcal{I}$
\[
F = \bigcup_{j \in \mathcal{I} : A_{i,j} = 1}  F_A^j 
\]
and so the result follows from Theorem \ref{main}.
\end{proof}

We note here that if the Hausdorff measure of a set is zero in a particular dimension, then the Hausdorff content is also zero in that dimension and so the equality is trivial.  One might initially wonder if $ \mathcal{H}^{\dim_\text{H}F} (F)=0$ always holds when $F$ is a self-similar set which cannot be defined via a system which satisfies the open set condition, but this is false, see for example \cite[Example 8.6]{farkas}.  Thus this result provides nontrivial information even when the open set condition is not satisfied.  Recall that Schief \cite{schief} proved that $ \mathcal{H}^{s} (F)=0$ if $F$ is a self-similar set defined via a system which does not satisfy the open set condition and $s$ is the \emph{similarity dimension} but, as the example of Farkas shows, one can obtain positive Hausdorff measure in the Hausdorff dimension if this is less than the similarity dimension, even if the open set condition cannot be satisfied.
\\ \\
A natural and important generalisation of self-similar sets is \emph{graph-directed self-similar sets}, which we now define.  Let $\Gamma = G(\mathcal{V}, \mathcal{E})$ be a finite strongly connected directed multigraph with vertices $\mathcal{V} = \{1, \dots, N\}$ and a finite multiset of edges $ \mathcal{E}$.  Write $\mathcal{E}_{i,j}$ for the multiset of all edges joining the vertex $i$ to the vertex $j$. For each $e \in \mathcal{E}$ associate a contracting similarity mapping $S_e$ on $\mathbb{R}^n$ with contraction ratio $r_e \in (0,1)$ which we again assume for convenience maps $[0,1]^n$ into itself.  It is standard that there exists a unique family of non-empty compact sets $\{F_{i}\}_{i  \in \mathcal{V}}$ satisfying
\begin{equation}
F_i = \bigcup_{j = 1}^{N} \bigcup_{e \in \mathcal{E}_{i,j}} S_e(F_j). \label{ssc-gda}
\end{equation}
Each set in the family $\{F_{i}\}_{i  \in \mathcal{V}}$ is called a graph-directed self-similar set.  Even though all self-similar sets are graph-directed self-similar sets, it was proved by Boore and Falconer \cite{boore} that graph-directed self-similar sets are genuinely more general than just self-similar sets.  We obtain the following generalisation of Corollary \ref{corselfsim}.

\begin{cor} \label{corgdselfsim}
Let $F \subseteq [0,1]^n$ be a graph-directed self-similar set and let $s = \dim_\text{\emph{H}} F$.  Then, regardless of separation conditions, $\mathcal{H}_\infty^s (F) =  \mathcal{H}^s (F)$.
\end{cor}

Corollary \ref{corgdselfsim} follows from Theorem \ref{main} and the following proposition.

\begin{prop} \label{prop_gda2ssoft}
Let $\{F_{i}\}_{i  \in \mathcal{V}}$ be the solution of a graph-directed self-similar iterated function system with directed graph $\Gamma = G(\mathcal{V}, \mathcal{E})$. Then there exists a subshift of finite type associated to the alphabet $\mathcal{I}=\mathcal{E}$ such that every $F_{i}$ is the union of 1-cylinders in the same family in the sense of Theorem \ref{main}. If $\Gamma$ is strongly connected then the constructed subshift of finite type is irreducible.
\end{prop}

\begin{proof}
Let the alphabet be indexed by the edge set $\mathcal{E}$. Now,  for two edges $e,f \in \mathcal{E}$, let $A_{e,f} = 1$ if and only if $f$ begins from the vertex where $e$ ended, i.e., it is possible to walk along $e$ and then along $f$. If $\Gamma$ is strongly connected, the matrix $A$ is irreducible.  It is now straightforward to see that for all $e \in \mathcal{E}_{i,j}$
\[
F_A^e = S_e(F_j)
\]
and so for all $i \in \mathcal{V}$ we have
\[
F_i = \bigcup_{j = 1}^{N} \bigcup_{e \in \mathcal{E}_{i,j}} F_A^e
\]
and, moreover, for any edge $e$ which finishes at $i$
\[
\bigcup_{j = 1}^{N} \mathcal{E}_{i,j} = \{ f \in  \mathcal{E} : A_{e,f} = 1\}
\]
as required.
\end{proof}

Proposition \ref{prop_gda2ssoft} says that the solution of every graph-directed self-similar iterated function system is a subshift of finite type in some sense. This is a folklore result and appears for example in \cite[Proposition 2.2.6]{lind}. The next proposition states that the converse is true which will be useful in Section \ref{packing_sec}.  Again this is a folklore result and appears for example in \cite[Proposition 2.3.9]{lind}.  We include both results and their simple proofs for completeness.

\begin{prop} \label{prop_ssoft2gda}
Let $\Sigma_{A}$ be a subshift of finite type for the alphabet $\mathcal{I}$ where $A$ has at least one non-zero entry in every row. Then there exits a graph-directed self-similar iterated function system with directed graph $\Gamma = G(\mathcal{I}, \mathcal{E})$ with solution $\{F_{A}^{i}\}_{i  \in \mathcal{I}}$. If $A$ is irreducible then $\Gamma$ is strongly connected.
\end{prop}

\begin{proof}
We draw a directed edge $e=e_{i,j}$ from $i$ to $j$ if $A_{i,j}=1$, let $S_{e}=S_{i}$ and let $\mathcal{E}=\{e_{i,j}:i,j \in \mathcal{I}, A_{i,j}=1\}$. If $A$ is irreducible then $\Gamma$ is strongly connected. We have that
\[
F_{A}^{i} \ =\bigcup_{j \in \mathcal{I},A_{i,j}=1} S_{i}\big(F_A^{j}\big) \ = \  \bigcup_{j \in \mathcal{I}} \bigcup_{e \in \mathcal{E}_{i,j}} S_e\big(F_A^j\big)
\]
and since there is a unique set of compact attractors associated to this graph-directed system, the proposition follows.
\end{proof}

\subsection{Extension to $k$-block subshifts of finite type}

We only consider 2-block subshifts of finite type in this paper, i.e. where the forbidden words are of length 2, but note that our results can be extended to the more general $k$-block case, where the forbidden words are of length $k$.  This is a natural simplification to make, as one can always reformulate a $k$-block subshift of finite type as a 2-block analogue over a larger alphabet, see \cite[Theorem 2.3.2]{lind}.  Moreover, this can be done so that the two systems are topologically conjugate which means that for irreducible $k$-block systems the associated $2$-block system remains irreducible.  The reformulation is straightforward and standard.  The new alphabet is the set of words of length $(k-1)$ such that there is an allowable word of length $k$ beginning with that word of length $(k-1)$.  Then, the 2-word (over the new alphabet) consisting of $(i_0, i_1, \dots, i_{k-2})$ followed by  $(i_1, i_2 ,\dots, i_{k-1})$ is allowed if and only if $(i_0, i_1, \dots, i_{k-1})$ was allowed in the original $k$-block system.  There is a naturally induced homeomorphism which conjugates the $k$-block system to the new $2$-block system.

\section{Ahlfors regularity and the weak separation property}

Our results have applications in studying Ahlfors regularity of self-similar sets and related fractals.  Recall that a bounded set $F \subseteq \mathbb{R}^n$ with Hausdorff dimension $s$ is called \emph{Ahlfors regular} if there exists a constant $c \geq 1$ such that for all $r \in (0, \text{diam}(F) ]$ and $x \in F$
\[
c^{-1} r^s \, \leq \, \mathcal{H}^s\big( F \cap B(x,r) \big)  \, \leq \, c r^s.
\]
It is straightforward to show that for an Ahlfors regular set the Hausdorff measure and Hausdorff content are equivalent in the Hausdorff dimension (equal up to a constant bound).  It is also well-known that a self-similar set satisfying the open set condition is Ahlfors regular.  Our results yield the following corollary.

\begin{cor} \label{ahlforscor}
 Let $A$ be irreducible and let $s = \dim_\text{\emph{H}} F_A$.  Then $\mathcal{H}^s (F_A) >0 $ if and only if $F_A$ is Ahlfors regular.  Moreover, this extends to any cylinder, i.e., for all $\textbf{i} \in \mathcal{I}^*$, $\mathcal{H}^s (F_A^\textbf{i}) >0 $ if and only if $F_A^\textbf{i}$ is Ahlfors regular.
\end{cor}

\begin{proof}
We will prove the result for $F_A$; the result for cylinders is similar and omitted.  Fix $r \in (0, \text{diam}(F_A) ]$ and $x \in F$.  The lower bound is straightforward and follows by choosing a first level cylinder with positive measure and then finding a copy of this cylinder inside $F_A \cap B(x,r)$ with diameter comparable to $r$ and then applying the scaling property for Hausdorff measure.  For the upper bound,
\begin{eqnarray*}
\mathcal{H}^s\big( F_A \cap B(x,r) \big)  & \leq  & \sum_{i \in \mathcal{I}} \mathcal{H}^s\big( F^i_A \cap B(x,r) \big)  \\ \\
& =  & \sum_{i \in \mathcal{I}} \mathcal{H}_\infty^s\big( F^i_A \cap B(x,r) \big) \qquad  \text{by Theorem \ref{main} and Lemma \ref{haus11}}\\ \\
& \leq   & \sum_{i \in \mathcal{I}} \text{diam}\big( F^i_A \cap B(x,r) \big)^s \\ \\
& \leq   & M(2r)^s
\end{eqnarray*}
completing the proof.
\end{proof}
Observe that the above corollary also applies to any collection of cylinders in $F_A$ and in particular to graph-directed self-similar sets.  Also, no separation conditions are assumed.
\\ \\
Let $F \subseteq \mathbb{R}^n$ be a self-similar set, not contained in any affine hyperplane. Recall that the \emph{weak separation property} is satisfied if the identity map is \emph{not} an accumulation point of the set
\[
 \{ S_{\textbf{\textbf{\emph{i}}}}^{-1} \circ S_{\textbf{\emph{j}}} : \textbf{\textbf{\emph{i}}}, \textbf{\textbf{\emph{j}}} \in \mathcal{I}^* \}
\]
equipped with the uniform norm, see \cite{zerner}.  It was shown in \cite[Theorem 2.1]{assouadoverlaps} that if $F$ satisfies the weak separation property (which is weaker than the open set condition) then it is Ahlfors regular.  It was also shown \cite[Theorem 1.4]{assouadoverlaps} that if $F$ does not satisfy the weak separation condition then the Assouad dimension $\dim_\text{A} F$ of $F$ is greater than or equal to 1.  In general the Assouad dimension is an upper bound for the Hausdorff dimension and we refer the reader to \cite{assouadoverlaps} for the definition.  This allows us to prove the following corollary.

\begin{cor} \label{ahlforscor2}
Let $F\subseteq \mathbb{R}^n$ be a self-similar set with Hausdorff dimension $s<1$ not contained in any affine hyperplane.  Then the following are equivalent:
\begin{itemize}
\item[(1)] $F$ satisfies the weak separation property
\item[(2)] $\mathcal{H}^s (F) >0$
\item[(3)] $0<\mathcal{H}^s (F) < \infty$
\item[(4)] $F$ is Ahlfors regular
\item[(5)] the Hausdorff and Assouad dimensions of $F$ coincide.
\end{itemize}
\end{cor}

\begin{proof}
Zerner \cite[Corollary after Proposition 2]{zerner} proved that (1) $\Rightarrow$ (2), (2) and (3) are equivalent since any self-similar set has finite Hausdorff measure in its Hausdorff dimension, see \cite[Corollary 3.3]{falconer}, our result, Corollary \ref{ahlforscor}, shows that (2) $\Leftrightarrow$ (4), the fact that (4) $\Rightarrow$ (5) is straightforward and folklore (see, for example, \cite[Proposition 2.1 (viii)]{tyson}), and since $\dim_\text{H} F < 1$ the result mentioned above \cite[Theorem 1.4]{assouadoverlaps} shows that (5) $\Rightarrow$ (1).
\end{proof}

The fact that (2) $\Rightarrow$ (1) provides a partial solution to a conjecture of Zerner, see the discussion following Proposition 2 in \cite{zerner}. We note that Corollary \ref{ahlforscor2} also shows that for self-similar sets with Hausdorff dimension strictly less than 1, the weak separation property can be formulated in a way which only depends on the set itself and not the defining iterated function system. The additional assumption $\dim_\text{H} F< 1$ required in the above corollary seems a little strange at first.  However, it turns out that this condition is sharp.  Firstly consider $F$ in the line.  It is straightforward to construct a self-similar set $F \subseteq [0,1]$ which fails the weak separation property, but for which $\mathcal{H}^1 (F) >0$.  For example, use the contractions $x \mapsto x/2$, $x \mapsto x/3$ and $x \mapsto x/2+1/2$ and apply the argument from \cite[Example 3.1]{fraser} using the fact that $\log 2/ \log 3 \notin \mathbb{Q}$.  We use a variation of this example to prove the following proposition demonstrating the (almost) sharpness of Corollary \ref{ahlforscor2}.
\begin{prop} \label{ahlforscor3}
For all $n \in \mathbb{N}\setminus \{1\}$ and all $s \in (1,n]$, there exists a self-similar set $F \subseteq [0,1]^n$ not contained in any affine hyperplane such that
\item[(1)] $F$ fails the weak separation property
\item[(2)] $\dim_\text{\emph{H}} F = s$
\item[(3)] $\mathcal{H}^s (F) >0$
\end{prop}

\begin{proof}
Let $r \in (0,1/2]$ be chosen such that
\[
\frac{\log 2}{-\log r} \ = \ \frac{s-1}{n-1} \  =: \  t
\]
and let $F_1 = [0,1]$ be viewed as a self-similar attractor of an iterated function system which fails the weak separation property and all of the maps have contraction ratio $r$.  Such an iterated function system can be constructed by modifying \cite[Section 2 (v)]{bandt-graf}.  Also, let $E \subseteq [0,1]$ be the self-similar set defined by the maps $x \mapsto rx$ and $x \mapsto rx+(1-r)$, and observe that $\dim_\text{H} E = t$ and $\mathcal{H}^t(E) >0$ since the open set condition is satisfied, see \cite[Corollary 3.3]{falconer}.  Now let $F = F_1 \times E^{n-1} \subseteq [0,1]^n$ be the product of $F_1$ with $n-1$ copies of $E$.  It is easy to see that $F$ is not contained in any affine hyperplane and that it is a self-similar set defined via the natural product iterated function system.  It follows from \cite[Theorem 8.10]{mattila} that $\dim_\text{H} F = 1 + (n-1) t = s$ and that $\mathcal{H}^s (F) >0$.  Note that to compute the dimension of $F$ here we used the fact that the Hausdorff and packing dimensions coincide for any self-similar set \cite[Corollary 3.3]{techniques}.  Finally it is easy to see that the weak separation property fails by virtue of it failing in the first coordinate.
\end{proof}

For $s=n$ in the above proposition our set $F$ is just $[0,1]^n$, which is not very interesting.  We point out that it is possible to construct a set with the desired properties but which has empty interior.  For example, it was shown in \cite{jordanetal} that there exists a self-similar set in the plane with positive $\mathcal{H}^2$ measure, but empty interior, and by \cite[Theorem 3]{zerner} such a set must fail the weak separation property.  We end this section by asking the natural question, an answer to which would complete the study.
\begin{ques}
Is it true that for all $n \in \mathbb{N}\setminus \{1\}$ there exists a self-similar set $F \subseteq [0,1]^n$ not contained in any affine hyperplane such that $F$ fails the weak separation property, $\dim_\text{\emph{H}} F = 1$ and $\mathcal{H}^1 (F) >0$?
\end{ques}

\section{Examples where $\mathcal{H}_\infty^s (F) <  \mathcal{H}^s (F) < \infty$ and future work} \label{examples}

In this section we give examples which show that equality of Hausdorff measure and Hausdorff content in the critical dimension is actually a rather special property.  In particular, we give several examples falling into natural classes of set for which one might hope to be able to extend Theorem \ref{main}, but for which equality does not hold.  A natural situation to consider is attractors of more general iterated function systems. In general an iterated function system (IFS) is a finite collection of contractions $\{S_i\}_{i \in \mathcal{I}}$ on a compact metric space. The \emph{attractor} of this system is the unique non-empty compact set $F$ satisfying
\[
F = \bigcup_{i \in \mathcal{I}} S_i (F).
\]
See \cite[Chapter 9]{falconer} and \cite{hutchinson} for more details on iterated function systems.  Two of the most standard and important generalisations of self-similar sets are \emph{self-affine sets}, where the defining maps are affine maps on some Euclidean space, and \emph{self-conformal sets}, where the defining maps are conformal.  We note that similarities are both affine and conformal.  It is evident that for any compact set $F \subset \mathbb{R}^n$ with Hausdorff dimension equal to 1, we have
\[
\mathcal{H}_\infty^1(F) \leq \text{diam}( F ).
\]
However, if $F$ is connected and not contained in a straight line, then 
\[
\mathcal{H}^1(F) > \text{diam}(F ).
\]
This phenomenon provides us with several simple counter examples.
\\ \\
\textbf{Self-affine sets:}  It was shown in \cite{bandt} that there exist self-affine curves $C$ in the plane which are differentiable at all but countably many points. In particular, these curves can have finite length but not lie in a straight line (see \cite[Example 10]{bandt} and \cite[Example 6.2]{kaenmaki}).  Such sets have Hausdorff dimension 1 and by the above argument satisfy
\[
0 < \mathcal{H}_\infty^1 (C) <  \mathcal{H}^1 (C) < \infty.
\]
\textbf{Self-conformal sets:} The upper half $A$ of the unit circle in the complex plane is a self-conformal set and has
\[
\mathcal{H}_\infty^1 (A) = 2 <  \pi =  \mathcal{H}^1 (A).
\]
The maps in the defining IFS for $A$ are $z \mapsto \sqrt{z}$ and $z \mapsto i \sqrt{z}$, defined on a suitable open domain containing $A$.
\\ \\
\textbf{Julia sets:}  the unit circle $S^1$ is the Julia set for the complex map $z \mapsto z^2$ and satisfies
\[
\mathcal{H}_\infty^1 (S^1) = 2 <  2 \pi =  \mathcal{H}^1 (S^1).
\]
\textbf{Sub-self-similar sets:}  Sub-self-similar sets, introduced by Falconer in \cite{subselfsim}, are compact sets $F$ satisfying
\[
F \  \subseteq  \  \bigcup_{i \in \mathcal{I}} S_i(F)
\]
for some IFS of similarities.  For any such IFS with the unit square as its attractor, the boundary of the unit square $Q = \partial [0,1]^2$ is a sub-self-similar set and satisfies
\[
\mathcal{H}_\infty^1 (Q) = \sqrt{2} <  4  =  \mathcal{H}^1 (Q).
\]
Finally we give two simple examples which show that Theorem \ref{main} is sharp, in some sense. 
\\ \\
\textbf{Non-irreducible subshift of finite type:} Consider the subshift of finite type on the alphabet $\{0,1,2\}$ given by the matrix
\[
A \ = \ \left ( \begin{array}{ccc}
1 & 0 & 0 \\ 
0 & 1 & 0 \\ 
1 & 1 & 0 \\ 
\end{array} \right ) 
\]
and associate any iterated function system consisting of three similarities on $[0,1]$ which map $[0,1]$ to three disjoint intervals.  Here $A$ is not irreducible and so does not fall into the class considered by Theorem \ref{main}.  The limit set $F = \Pi\big(\Sigma_A\big)$ consists of only four points and so $F$ and all of its children have Hausdorff dimension 0, but nevertheless
\[
\mathcal{H}_\infty^0 \big(F_A^2\big) =  1 <  2 =  \mathcal{H}^0 \big(F_A^2\big).
\]
\textbf{Full set for irreducible and aperiodic subshift of finite type:} Now we will show that one cannot hope to have $\mathcal{H}_\infty^s (F_A)  =  \mathcal{H}^s (F_A)$ for even an \emph{aperiodic} subshift of finite type (which we recall is a stronger condition than irreducible).  Consider the alphabet $\{0,1,2,3\}$ and let
\[
A \ = \ \left ( \begin{array}{cccc}
1 & 1 & 0 & 0 \\ 
0 & 0 & 1 & 1 \\  
0 & 0 & 1 & 1 \\ 
1 & 1 & 0 & 0 \\
\end{array} \right ) 
\]
which is quickly seen to be aperiodic.  Define similarities on the unit square by
\[
S_0(x,y) = (x/2,y/2), \qquad \qquad S_1(x,y) = (-x/2,y/2) + (1/2, 1/2), 
\]
\[
S_2(x,y) = (x/2,y/2) + (1/2,0), \quad \text{and}  \quad S_3(x,y) = (-x/2,y/2) + (1, 1/2).
\]
It is easy to see that
\[
F_A =( \{0\} \times [0,1]) \cup (\{1\} \times [0,1])
\]
which satisfies
\[
 \mathcal{H}_\infty^1 (F_A) = \sqrt{2} <  2 =  \mathcal{H}^1 (F_A).
\]
Of course Theorem \ref{main} still correctly states that
\[
 \mathcal{H}^1_\infty \big(F_A^0 \cup F_A^1\big) =   \mathcal{H}^1 \big(F_A^0 \cup F_A^1\big)  \quad \text{and}  \quad  \mathcal{H}_\infty^1 \big(F_A^2 \cup F_A^3\big) =   \mathcal{H}^1\big(F_A^2 \cup F_A^3\big),
\]
noting that
\[
F_A^0 \cup F_A^1 = \{0\} \times [0,1] \quad \text{and}  \quad  F_A^2 \cup F_A^3 = \{1\} \times [0,1].
\]
A possible direction for further study on this topic would be to consider the classes of sets studied in this section, namely, self-conformal, self-affine, sub-self-similar, or Julia sets, and try to prove that the Hausdorff measure and Hausdorff content agree in some interesting subclass.  Alternatively, one could look for negative results, which prove that the Hausdorff measure and Hausdorff content are always distinct in certain subclasses.   Also, all of our counter examples in these classes were using sets with dimension 1.  Could there be different phenomena at work for non-integral dimensions?  We suspect not, but have not investigated this further.  Note that we cannot give a simple condition guaranteeing $\mathcal{H}_\infty^s(F) < \mathcal{H}^s(F)$ apart from for connected sets $F$ not lying in a straight line with Hausdorff dimension $s=1$.  This is because such sets may be $s$-straight by the result of Delaware mentioned previously \cite{delaware}.

\section{The question of packing measure} \label{packing_sec}

In this section we address the question of whether analogous results can be obtained for packing measure and a suitably defined `packing content'.  First we recall the definition of the packing measure.  Packing measure, defined in terms of \emph{packings}, is a natural dual to Hausdorff measure, which was defined in terms of \emph{covers}.  For $s \geq 0$ and $\delta>0$ the $\delta$-approximate $s$-dimensional packing pre-measure of $F$ is defined by
\begin{eqnarray*}
\mathcal{P}_{\delta}^s (F) = \sup \bigg\{ \sum_{k=1}^{\infty} \text{diam}(U_k )^s : \{ U_k \}_{k=1}^{\infty} \text{ is a countable collection of pairwise disjoint } \\ 
\text{closed balls centered in $F$  with $\text{diam}(U_k) \leq \delta$ for all $k$} \bigg\}
\end{eqnarray*}
and the $s$-dimensional packing pre-measure of $F$ by $\mathcal{P}_0^s (F) = \lim_{\delta \to 0} \mathcal{P}_{\delta}^s (F)$.  To ensure countable subadditivity, the packing (outer) measure of $F$ is defined by
\[
\mathcal{P}^s (F) = \inf \bigg\{\sum_i \mathcal{P}_0^s (F_i) : F \subseteq \bigcup_i F_i  \bigg\}.
\]
It follows from the definition that
\begin{equation}
 \mathcal{P}^s (F)\leq \mathcal{P}_{0}^s (F)\leq \mathcal{P}_{\delta}^s (F).\label{packineq}
\end{equation}
Similar to the Hausdorff dimension, the packing dimension of $F$ is defined to be
\[
\dim_\text{P} F = \inf \Big\{ s \geq 0: \mathcal{P}^s (F) =0 \Big\}.
\]

The extra step in the definition of packing measure makes it often more difficult to handle than the Hausdorff measure.  However, in our setting there is a useful simplification due to Feng-Hua-Wen \cite{packingmeasure} and Haase \cite{haase}.

\begin{prop}
Let $F$ be a compact subset of $\mathbb{R}^n$ with the property that for every open ball $B$ centered in $F$, there exists a bi-Lipschitz map $S$ on $\mathbb{R}^n$ such that $S(F) \subseteq B \cap F$.  Then for all $s \geq 0$ we have
\[
\mathcal{P}^s (F) \ = \ \mathcal{P}_0^s (F).
\]
\end{prop}

\begin{proof}
For any compact set $F \subset \mathbb{R}^n$, if $\mathcal{P}^s_0 (F)  < \infty$, then $\mathcal{P}^s (F) \ = \ \mathcal{P}_0^s (F)$, by the main result in \cite{packingmeasure}.  In the case when $\mathcal{P}^s_0 (F)  = \infty$, the additional assumption implies that $\mathcal{P}^s_0 (B \cap F) = \infty$ for all open balls intersecting $F$, which by \cite[Lemma 4]{haase}, implies that $\mathcal{P}^s(F)  = \infty$.
\end{proof}

For this reason we can concern ourselves only with the packing pre-measure, which is easier to understand.  The first question is, how should we define the packing (pre) content?  If we naively define it by just removing the bounds on the diameters of the balls in the packing, then the answer is always infinity, as long as $s>0$ and $F\neq\emptyset$.  This is because one can just take a packing by a single ball with unbounded diameter.  Possible alternatives would be either to insist that there are at least two balls in every packing, or to bound the radii by something concrete, such as the diameter of $F$ itself.  However, it might be more natural to try to prove that for sufficiently small $\delta$, the equality  $\mathcal{P}_0^s (F) \ = \ \mathcal{P}_{\delta}^s (F)$ is satisfied.  We adopt this third approach.  The next question is, do we expect this to be true in the same setting as Theorem \ref{main}? An archetypal question being:
\\ \\
``\emph{If $F$ is self-similar, then does there exists a $\delta_0>0$ such that for all $\delta \in (0, \delta_0)$ we have}
\[
 \mathcal{P}^s (F) \ = \  \mathcal{P}_0^s (F) \ = \ \mathcal{P}_{\delta}^s (F)?\text{''}
\]
One strange consequence of this would be that for such sets the packing measure is always strictly positive.  In the same way that $\mathcal{H}^s_\delta(F)$ is always finite for bounded sets, we have that $\mathcal{P}^s_{\delta}(F)$ is always positive for arbitrary non-empty sets.  Interestingly enough it was an important question for about 15 years whether or not it was possible for a self-similar set to have zero packing measure in its dimension, see \cite{problems_update}, but this was recently resolved by Orponen \cite{orponen}, who provided a family of self-similar sets for whose elements $F$ (of course not satisfying the open set condition) $ \mathcal{P}^{\dim_\P F} (F) = 0$.  Thus the answer to the above question is immediately `no'.  We have managed to prove a weaker result, however, which we state after briefly recalling the \emph{strong separation condition}.  This is a strictly stronger condition than the open set condition and is satisfied if the images of the attractor under the maps in the defining system are pairwise disjoint.  We also recall that for any self-similar set, the packing measure must be finite in the packing dimension, see \cite[Exercise 3.2]{techniques}.

\begin{thm} \label{packing}
Let $F \subseteq \mathbb{R}^n$ be a self-similar set which satisfies the strong separation condition and let $s = \dim_\text{\emph{P}} F$.  Then, there exists a $\delta_0>0$ such that for all $\delta \in (0, \delta_0)$ we have
\[
0< \mathcal{P}^s (F) \ = \  \mathcal{P}_0^s (F) \ = \ \mathcal{P}_{\delta}^s (F)< \infty.
\]
\end{thm}

We will prove Theorem \ref{packing} in Section \ref{packingproof}. By the above discussion, this result does not extend to $F$ which do not satisfy the open set condition. It is also easy to see that it does not extend to the open set condition case either.  For example, the unit interval $I$ is a self-similar set satisfying the open set condition but not the strong separation condition.  Elementary calculations yield that $\mathcal{P}^1(I) = 1$, but that $\mathcal{P}_\delta^1 (I)= 1+\delta$ for all $\delta$.  We pose the question of whether the appropriate converse of Theorem \ref{packing} is true.

\begin{ques}
Does there exists a self-similar set $F$ satisfying the open set condition, but for which there is no IFS of similarities satisfying the strong separation condition with $F$ as the attractor, for which there exists a $\delta_0>0$ such that for all $\delta \in (0, \delta_0)$ we have
\[
0< \mathcal{P}^s (F) \ = \  \mathcal{P}_0^s (F) \ = \ \mathcal{P}_{\delta}^s (F)< \infty?
\]
\end{ques}

We generalise Theorem \ref{packing} for graph-directed self-similar sets and subshifts of finite type. A graph-directed self-similar iterated function system satisfies the \emph{strong separation condition} if (\ref{ssc-gda}) is a disjoint union for every $i$.

\begin{thm} \label{packing-gda}
Let $\{F_{i}\}_{i  \in \mathcal{V}}$ be the solution of a graph-directed self-similar iterated function system which satisfies the strong separation condition and let $s$ be the common packing dimension of the sets $\{F_i\}_{i \in \mathcal{V}}$.  Then, there exists a $\delta_0>0$ such that for all $\delta \in (0, \delta_0)$ and all $i\in \mathcal{V}$ we have
\[
0< \mathcal{P}^s (F_i) \ = \  \mathcal{P}_0^s (F_i) \ = \ \mathcal{P}_{\delta}^s (F_i)< \infty.
\]
\end{thm}

We will prove Theorem \ref{packing-gda} in Section \ref{packingproof}. Due to Proposition \ref{prop_ssoft2gda} it follows that this result generalises to subshift of finite types.

\begin{thm} \label{packing-ssoft}
Let $\Sigma_{A}$ be an irreducible subshift of finite type on the alphabet $\mathcal{I}$ and let $s = \dim_\text{\emph{P}} F_A$. Assume that

\begin{equation} \label{ssoft-ssc}
\{F_A^j\}_{j \in \mathcal{I} : A_{i,j} = 1}
\end{equation}

are disjoint for every $i\in \mathcal{I}$. Then, there exists a $\delta_0>0$ such that for all $\delta \in (0, \delta_0)$ and all $i\in \mathcal{I}$ we have
\[
0< \mathcal{P}^s \big(F_A^i \big) \ = \  \mathcal{P}_0^s  \big(F_A^i \big) \ = \ \mathcal{P}_{\delta}^s  \big(F_A^i \big)< \infty.
\]
\end{thm}
Theorem \ref{packing-ssoft} follows from Theorem \ref{packing-gda} and Proposition \ref{prop_ssoft2gda} since (\ref{ssoft-ssc}) ensures that the strong separation condition is satisfied for the graph-directed system in  Proposition \ref{prop_ssoft2gda}.

\section{Proofs}

\subsection{A useful exhaustion lemma}

In this section we prove an exhaustion lemma for Hausdorff measure, similar to \cite[Proposition 1.9]{farkas}, which may be of interest in its own right.  It shows that we can exhaust the Hausdorff measure of a (potentially overlapping) subset of a self-similar set modelled by a subshift of finite type by infinitely many, disjoint, images of first level cylinders. First we state a version of Vitali's covering theorem. Let $H\subset\mathbb{R}^{n}$. A collection of sets $\mathcal{A}$
is called a \textsl{Vitali cover} of $H$ if for each $x\in H$ and $\delta>0$
there exist $A\in\mathcal{A}$ with $x\in A$ and $0<\mathrm{diam}(A)<\delta$.

\begin{prop}
\label{lem:VCT}
Let $H\subset\mathbb{R}^{n}$ be a $\mathcal{H}^{s}$-measurable set
with $\mathcal{H}^{s}(H)<\infty$ and $B_{1},\ldots,B_{m}\subset\mathbb{R}^{n}$
be closed sets with $0<\mathrm{diam}(B_{i})<\infty$ and $0<\mathcal{H}^{s}(B_{i})<\infty$
for all $i\in\left\{ 1,\ldots,m\right\} $. Let $\mathcal{A}$ be
a Vitali cover of $H$ such that every element of $\mathcal{A}$ is
similar to $B_{i}$ for some $i\in\left\{ 1,\ldots,m\right\} $ and
every element of $\mathcal{A}$ is a subset of $H$. Then there exists
a disjoint sequence of sets (finite or countable) $A_{1},A_{2},\ldots\in\mathcal{A}$
such that $\mathcal{H}^{s}\left(H\setminus\left(\bigcup_{i=1}^{\infty}A_{i}\right)\right)=0$.
\end{prop}
\begin{proof}
Assume that $A_{1},A_{2},\ldots\in\mathcal{A}$ is a disjoint sequence
of sets. Let $M=\max_{1\leq i\leq m}\frac{\mathrm{diam}(B_{i})^{s}}{\mathcal{H}^{s}(B_{i})}$.
If $A_{i}$ is similar to $B_{j}$ then
\[
\mathrm{diam}(A_{i})^{s} \ = \ \mathcal{H}^{s}(A_{i})\, \frac{\mathrm{diam}(B_{j})^{s}}{\mathcal{H}^{s}(B_{j})} \ \leq \ \mathcal{H}^{s}(A_{i})\, M.
\]
Hence
\[
\sum_{i=1}^{\infty}\mathrm{diam}(A_{i})^{s} \ = \ \sum_{i=1}^{\infty}\mathcal{H}^{s}(A_{i})\, M \ \leq \ \mathcal{H}^{s}(H)\, M \ < \ \infty.
\]
Thus the proposition follows from a version of Vitali's covering theorem \cite[Theorem 1.10]{vitali}.
\end{proof}
Let
\[
\mathcal{I}_{A}^{*}=\left\{ \textbf{\emph{i}}\in\mathcal{I}^{*}:\exists\alpha\in\Sigma_{A}\, \text{and} \, k\in\mathbb{N}\, \text{such that} \,\alpha\vert_{k}=\textbf{\emph{i}}\right\}
\]
and for $\textbf{\emph{i}}\in\mathcal{I}_{A}^{*}$ let
\[
\mathcal{I}_{A}^{\textbf{\emph{i}}*}=\left\{ \textbf{\emph{j}}\in\mathcal{I}_{A}^{*}:\textbf{\emph{j}}\vert_{\left|\textbf{\emph{i}}\right|}=\textbf{\emph{i}}\right\}.
\]
For $\textbf{\emph{i}}=(i_{0},\ldots,i_{k-1})\in\mathcal{I}^{*}$
with $\left|\textbf{\emph{i}}\right|\geq1$ we define $(\textbf{\emph{i}})_{0}=i_{0}$
and $(\textbf{\emph{i}})_{last}=i_{k-1}$ and $\tau(\textbf{\emph{i}})=(i_{0},\ldots,i_{k-2})$.
If $\textbf{\emph{i}}=(i_{0},\ldots,i_{k-1}),\textbf{\emph{j}}=(j_{0},\ldots,j_{l-1})\in\mathcal{I}_{A}^{*}$ are
such that $A_{(\textbf{\emph{i}})_{last},(\textbf{\emph{j}})_{0}}=1$
then we write $\textbf{\emph{i}}*\textbf{\emph{j}}=(i_{0},\ldots,i_{k-1},j_{0},\ldots,j_{l-1})\in\mathcal{I}_{A}^{*}$.

\begin{lma} \label{exhaust}
Let $A$ be an irreducible subshift of finite type, let $s=\dim_{\emph{H}}F_{A}$
and assume that $\mathcal{H}^{s}(F_{A})>0$. Then for each $j\in\mathcal{I}$,
there exists a collection $\mathcal{I}_{\infty}^{j}$ of finite words
$\textbf{i}\in\mathcal{I}^{*}$ that satisfies the following properties:

\noindent (i) the first symbol is $j$, i.e. $(\textbf{i})_0=j$,

\noindent (ii) the last symbol is $j$, i.e. $(\textbf{i})_{last}=j$,

\noindent (iii) there exists $\alpha\in\Sigma_{A}$ and $k\in\mathbb{N}$ such
that $\alpha\vert_{k}=\textbf{i}$ or, in other words, $\textbf{i}\in\mathcal{I}^{*}_A$,

\noindent (iv) for $\textbf{i},\textbf{j}\in\mathcal{I}_{\infty}^{j}$
with $\textbf{i}\neq\textbf{j}$ we have
that
\[
F_{A}^{\textbf{i}}\cap F_{A}^{\textbf{j}}=\emptyset,
\]
 (v) 
\[
\mathcal{H}^{s}\left(F_{A}^{j}\setminus\left(\bigcup_{\textbf{i}\in\mathcal{I}_{\infty}^{j}}F_{A}^{\textbf{i}}\right)\right)=0,
\]

\noindent (vi) the contraction ratios satisfy a Hutchinson-Moran type
expression for the Hausdorff dimension, i.e.
\[
\sum_{\textbf{i}\in\mathcal{I}_{\infty}^{j}}r_{\tau(\textbf{i})}^{s}=1.
\]

\end{lma}

\begin{proof}
Since $A$ is irreducible, for every $i\in\mathcal{I}\setminus\left\{ j\right\} $
we can find $\textbf{\emph{i}}_{i}\in\mathcal{I}_{A}^{i*}$ such
that $(\textbf{\emph{i}}_{i})_{0}=i$ and $(\textbf{\emph{i}}_{i})_{last}=j$.  Of course there are infinitely many such choices for $\textbf{\emph{i}}_{i}$, but for definiteness choose one with minimal length.
Thus if $\textbf{\emph{i}}\in\mathcal{I}_{A}^{j*}$ and $(\textbf{\emph{i}})_{last}=i$
then $\tau(\textbf{\emph{i}})*\textbf{\emph{i}}_{i}\in\mathcal{I}_{A}^{j*}$
and $(\tau(\textbf{\emph{i}})*\textbf{\emph{i}}_{i})_{last}=j$.
\foreignlanguage{english}{Let
\begin{equation}
r_{\mathrm{\mathrm{min}}}=\min\left\{ \frac{\mathcal{H}^{s}\left(F_{A}^{\textbf{\emph{i}}_{i}}\right)}{\mathcal{H}^{s}\left(F_{A}^{i}\right)}:i\in\mathcal{I}\setminus\left\{ j\right\} \right\} \in(0,1).\label{eq:r_min}
\end{equation}
}

We define a sequence $\mathcal{I}_{0}^{j},\mathcal{I}_{1}^{j},\ldots$
inductively where $\mathcal{I}_{n}^{j}$ satisfies properties \textit{(i),
(iii), (iv)} and \textit{(v)}. The collection of sets $\left\{ F_{A}^{\textbf{\emph{i}}}:\textbf{\emph{i}}\in\mathcal{I}_{A}^{j*},\left|\textbf{\emph{i}}\right|\geq2\right\} $
is a Vitali cover of $F_{A}^{j}$ and hence by Proposition \ref{lem:VCT}
there exists $\mathcal{I}_{0}^{j}\subseteq\left\{ \textbf{\emph{i}}\in\mathcal{I}_{A}^{j*}:\left|\textbf{\emph{i}}\right|\geq2\right\} $
 such that $F_{A}^{\textbf{\emph{i}}}\cap F_{A}^{\textbf{\emph{j}}}=\emptyset$
for $\textbf{\emph{i}},\textbf{\emph{j}}\in\mathcal{I}_{0}^{j}$,
$\textbf{\emph{i}}\neq\textbf{\emph{j}}$ and 
\[
\mathcal{H}^{s}\left(F_{A}^{j}\setminus\left(\bigcup_{\textbf{\emph{i}}\in\mathcal{I}_{0}^{j}}F_{A}^{\textbf{\emph{i}}}\right)\right)=0.
\]
Once $\mathcal{I}_{n}^{j}$ is defined we define $\mathcal{I}_{n+1}^{j}$
as follows. First, for each $\textbf{\emph{i}}\in\mathcal{I}_{n}^{j}$
we define a set $\mathcal{I}_{n+1,\textbf{\emph{i}}}^{j}$.
If $(\textbf{\emph{i}})_{last}=j$ then $\mathcal{I}_{n+1,\textbf{\emph{i}}}^{j}=\left\{ \textbf{\emph{i}}\right\} $.
If $(\textbf{\emph{i}})_{last}=i\neq j$
then
\[
\left\{ F_{A}^{\tau(\textbf{\emph{i}})*\textbf{\emph{j}}}:\textbf{\emph{j}}\in\mathcal{I}_{A}^{i*},F_{A}^{\tau(\textbf{\emph{i}})*\textbf{\emph{j}}}\cap F_{A}^{\tau(\textbf{\emph{i}})*\textbf{\emph{j}}_{i}}=\emptyset\right\}
\]
is a Vitali cover of $F_{A}^{\textbf{\emph{i}}}\setminus F_{A}^{\tau(\textbf{\emph{i}})*\textbf{\emph{j}}_{i}}$ and
hence by Proposition \ref{lem:VCT} there exists 
\[
\mathcal{J}_{n+1,\textbf{\emph{i}}} \ \subseteq \ \left\{ \textbf{\emph{j}}:\textbf{\emph{j}}\in\mathcal{I}_{A}^{i*},F_{A}^{\tau(\textbf{\emph{i}})*\textbf{\emph{j}}}\cap F_{A}^{\tau(\textbf{\emph{i}})*\textbf{\emph{j}}_{i}}=\emptyset\right\}
\]
such that $F_{A}^{\tau(\textbf{\emph{i}})*\textbf{\emph{i}}_{1}}\cap F_{A}^{\tau(\textbf{\emph{i}})*\textbf{\emph{i}}_{2}}=\emptyset$
for all $\textbf{\emph{i}}_{1},\textbf{\emph{i}}_{2}\in\mathcal{J}_{n+1,\textbf{\emph{i}}}^{j}$, with 
$\textbf{\emph{i}}_{1}\neq\textbf{\emph{i}}_{2}$, and
\[
\mathcal{H}^{s}\left(\left(F_{A}^{\textbf{\emph{i}}}\setminus F_{A}^{\tau(\textbf{\emph{i}})*\textbf{\emph{j}}_{i}}\right)\setminus\left(\bigcup_{\textbf{\emph{j}}\in\mathcal{J}_{n+1,\textbf{\emph{i}}}^{j}}F_{A}^{\tau(\textbf{\emph{i}})*\textbf{\emph{j}}}\right)\right) \ = \ 0.
\]
Now let
\[
\mathcal{I}_{n+1,\textbf{\emph{i}}}^{j} \ = \ \left\{ \tau(\textbf{\emph{i}})*\textbf{\emph{j}}_{i}\right\} \, \cup \, \left\{ \tau(\textbf{\emph{i}})*\textbf{\emph{j}}:\textbf{\emph{j}}\in\mathcal{J}_{n+1,\textbf{\emph{i}}}^{j}\right\}
\]
and
\[
\mathcal{I}_{n+1}^{j} \ = \ \bigcup_{\textbf{\emph{i}}\in\mathcal{I}_{n}}\mathcal{I}_{n+1,\textbf{\emph{i}}}^{j}.
\]
\selectlanguage{english}%
Finally we define
\[
\mathcal{I}_{\infty}^{j} \ = \ \bigcap_{n_{1}=1}^{\infty}\bigcup_{n_{2}=n_{1}}^{\infty}\mathcal{I}_{n_{2}}^{j}.
\]
Clearly $F_{A}^{\textbf{\emph{i}}}\cap F_{A}^{\textbf{\emph{j}}}=\emptyset$
for $\textbf{\emph{i}},\textbf{\emph{j}}\in\mathcal{I}_{\infty}^{j}$,
$\textbf{\emph{i}}\neq\textbf{\emph{j}}$. If $\textbf{\emph{i}}\in\mathcal{I}_{n}^{j}$
and $(\textbf{\emph{i}})_{last}\neq j$ then $\textbf{\emph{i}}\notin\mathcal{I}_{n+l}^{j}$
for every positive integer $l$, hence $\textbf{\emph{i}}\notin\mathcal{I}_{\infty}^{j}$.
So $(\textbf{\emph{i}})_{last}=j$ for all $\textbf{\emph{i}}\in\mathcal{I}_{\infty}^{j}$.
Clearly
\begin{equation}
\mathcal{H}^{s}\left(F_{A}^{j}\setminus\left(\bigcup_{\textbf{\emph{i}}\in\mathcal{I}_{n}^{j}}F_{A}^{\textbf{\emph{i}}}\right)\right) \ = \ 0\label{eq:Egyenlito lem 1}
\end{equation}
for every positive integer $n$. For $\textbf{\emph{i}}\in\mathcal{I}_{n}^{j}$
such that $(\textbf{\emph{i}})_{last}=i\neq j$ we have that
\[
\left\{ \textbf{\emph{j}}:\tau(\textbf{\emph{i}})*\textbf{\emph{j}}\in\mathcal{I}_{n+1},(\tau(\textbf{\emph{i}})*\textbf{\emph{j}})_{last}\neq j\right\}  \ \subseteq \ \mathcal{J}_{n+1,\textbf{\emph{i}}}^{j}
\]
and
\begin{eqnarray}
\mathcal{H}^{s} \left(F_{A}^{\tau(\textbf{\emph{i}})*\textbf{\emph{j}}_{i}}\right) \ \ =  \ \  \mathcal{H}^{s}\left(S_{\tau(\textbf{\emph{i}})*\tau(\textbf{\emph{j}}_{i})}(F_{A}^{j})\right)&=& r_{\tau(\textbf{\emph{i}})}^{s}r_{\tau(\textbf{\emph{j}}_{i})}^{s}\mathcal{H}^{s}\left(F_{A}^{j}\right)\frac{\mathcal{H}^{s}\left(F_{A}^{i}\right)}{\mathcal{H}^{s}\left(F_{A}^{i}\right)} \nonumber \\ \nonumber \\
&=&\mathcal{H}^{s}\left(F_{A}^{\textbf{\emph{j}}_{i}}\right)\frac{\mathcal{H}^{s}\left(F_{A}^{\textbf{\emph{i}}}\right)}{\mathcal{H}^{s}\left(F_{A}^{i}\right)} \nonumber \\ \nonumber \\
&\geq& r_{\mathrm{\min}}\mathcal{H}^{s}\left(F_{A}^{\textbf{\emph{i}}}\right) \label{lowerestimate1}
\end{eqnarray}
by (\ref{eq:r_min}).  Also $(\tau(\textbf{\emph{i}})*\textbf{\emph{j}}_{i})_{last}=j$ by definition. Therefore
\[
\mathcal{I}_{n+1}^{j}\setminus\mathcal{I}_{\infty}^{j} \ \subseteq \ \bigcup_{\textbf{\emph{i}}\in\mathcal{I}_{n}^{j}\setminus\mathcal{I}_{\infty}^{j}}\left\{ \tau(\textbf{\emph{i}})*\textbf{\emph{j}}:\textbf{\emph{j}}\in\mathcal{J}_{n+1,\textbf{\emph{i}}}^{j}\right\}
\]
and 
\begin{eqnarray*}
\mathcal{H}^{s}\left(\bigcup_{\textbf{\emph{i}}\in\mathcal{I}_{n+1}^{j}\setminus\mathcal{I}_{\infty}^{j}}F_{A}^{\textbf{\emph{i}}}\right) &\leq& \sum_{\textbf{\emph{i}}\in\mathcal{I}_{n}^{j}\setminus\mathcal{I}_{\infty}^{j}}\mathcal{H}^{s}\left(\bigcup_{\textbf{\emph{j}}\in\mathcal{J}_{n+1,\textbf{\emph{i}}}^{j}}F_{A}^{\tau(\textbf{\emph{i}})*\textbf{\emph{j}}}\right)\\ \\
&\leq& \sum_{\textbf{\emph{i}}\in\mathcal{I}_{n}^{j}\setminus\mathcal{I}_{\infty}^{j}}\mathcal{H}^{s}\left( F_{A}^{\textbf{\emph{i}}}  \setminus   F_{A}^{\tau(\textbf{\emph{i}})*\textbf{\emph{j}}_{i}}    \right)\\ \\
&\leq& \sum_{\textbf{\emph{i}}\in\mathcal{I}_{n}^{j}\setminus\mathcal{I}_{\infty}^{j}}\left(\mathcal{H}^{s}\left(F_{A}^{\textbf{\emph{i}}}\right)-r_{\mathrm{min}}\mathcal{H}^{s}\left(F_{A}^{\textbf{\emph{i}}}\right)\right) \qquad \text{ by (\ref{lowerestimate1})}\\ \\
&=&\sum_{\textbf{\emph{i}}\in\mathcal{I}_{n}^{j}\setminus\mathcal{I}_{\infty}^{j}}(1-r_{\mathrm{min}})\, \mathcal{H}^{s}\left(F_{A}^{\textbf{\emph{i}}}\right) \\ \\
&=&(1-r_{\mathrm{min}})\, \mathcal{H}^{s}\left(\bigcup_{\textbf{\emph{i}}\in\mathcal{I}_{n}^{j}\setminus\mathcal{I}_{\infty}^{j}}F_{A}^{\textbf{\emph{i}}}\right).
\end{eqnarray*}
Hence
\[
\mathcal{H}^{s}\left(\bigcup_{\textbf{\emph{i}}\in\mathcal{I}_{n+1}^{j}\setminus\mathcal{I}_{\infty}^{j}}F_{A}^{\textbf{\emph{i}}}\right) \ \leq \ (1-r_{\mathrm{min}})^{n}\, \mathcal{H}^{s}\left(\bigcup_{\textbf{\emph{i}}\in\mathcal{I}_{0}^{j}\setminus\mathcal{I}_{\infty}^{j}}F_{A}^{\textbf{\emph{i}}}\right) \ = \ (1-r_{\mathrm{min}})^{n}\, \mathcal{H}^{s}\left(F_{A}^{j}\right)
\]
for all $n\in\mathbb{N}$ and combined with (\ref{eq:Egyenlito lem 1})
we get that
\[
\mathcal{H}^{s}\left(\bigcup_{\textbf{\emph{i}}\in\mathcal{I}_{n+1}^{j}\cap\mathcal{I}_{\infty}^{j}}F_{A}^{\textbf{\emph{i}}}\right) \ \geq \ \big(1-(1-r_{\mathrm{min}})^{n}\big)\, \mathcal{H}^{s}\left(F_{A}^{j}\right).
\]
Thus 
\[
\mathcal{H}^{s}\left(\bigcup_{\textbf{\emph{i}}\in\mathcal{I}_{\infty}^{j}}F_{A}^{\textbf{\emph{i}}}\right) \ \geq \ \mathcal{H}^{s}\left(F_{A}^{j}\right)
\]
and so \[
\mathcal{H}^{s}\left(F_{A}^{j}\setminus\left(\bigcup_{\textbf{\emph{i}}\in\mathcal{I}_{\infty}^{j}}F_{A}^{\textbf{\emph{i}}}\right)\right) \ = \ 0.
\]

\selectlanguage{british}%
Thus the collection $\mathcal{I}_{\infty}^{j}$ satisfies properties
\textit{(i)-(v)}. Property \textit{(vi)} follows easily from \textit{(iv)}
and \textit{(v)} since
\[
\mathcal{H}^{s}\left(F_{A}^{j}\right) \ = \ \sum_{\textbf{\emph{i}}\in\mathcal{I}_{\infty}^{j}}\mathcal{H}^{s}\left(F_{A}^{\textbf{\emph{i}}}\right) \ = \ \sum_{\textbf{\emph{i}}\in\mathcal{I}_{\infty}^{j}}\mathcal{H}^{s}\left(S_{\tau(\textbf{\emph{i}})}\left(F_{A}^{j}\right)\right) \ = \ \sum_{\textbf{\emph{i}}\in\mathcal{I}_{\infty}^{j}}r_{\tau(\textbf{\emph{i}})}^{s}\mathcal{H}^{s}\left(F_{A}^{j}\right)
\]
and the fact that we can divide by $\mathcal{H}^{s}\left(F_{A}^{j}\right)$.
\end{proof}

\subsection{Proof of Theorem \ref{main}} \label{mainproof}

In this section we will prove our main result.  It is trivially true if $\mathcal{H}^s(F_A) = 0$, so we assume otherwise.  Fix $i \in \mathcal{I}$ and $\varepsilon>0$.  Choose a countable open cover $\{U_k\}_{k \in \mathcal{K}}$ of $F_A^i$ which satisfies
\begin{equation} \label{111}
 \sum_{k \in \mathcal{K}} \text{diam}( U_k )^s \ \leq  \ \mathcal{H}^s_\infty(F_A^i) \, + \, \varepsilon.
\end{equation}
Since $F_A^i$ is bounded we can assume that there is a uniform bound on the diameters of the $U_k$.  Let $\mathcal{I}^{i}_\infty$ be the `exhausting set' from Lemma \ref{exhaust}.  For $m \in \mathbb{N}$, let
\begin{equation} \label{abab}
\mathcal{I}^{i,m}_\infty \ = \ \Big\{ \textbf{\emph{i}}' \in \mathcal{I}^* : \textbf{\emph{i}}' = \tau(\textbf{\emph{i}}^0)  \tau(\textbf{\emph{i}}^1) \dots  \tau(\textbf{\emph{i}}^{m-1}) \text{ where } \textbf{\emph{i}}^l \in \mathcal{I}^{i}_\infty \text{ for } l=0, \dots, m-1 \Big\}.
\end{equation}
By properties \emph{(i)} and \emph{(ii)} in Lemma \ref{exhaust} the set $\mathcal{I}^{i,m}_\infty$ is a set of restricted words from $\Sigma^{i}_A$.  Moreover, it follows form property \emph{(v)} in Lemma \ref{exhaust} that, for all $m \in \mathbb{N}$,
\begin{equation} \label{222}
 \mathcal{H}^s \left(F_A^i \setminus \bigcup_{\textbf{\emph{i}} \in \mathcal{I}^{i,m}_\infty} S_\textbf{\emph{i}}(F_A^i) \right)  \ = \ 0.
\end{equation}
Observe that, for each $m \in \mathbb{N}$, 
\[
\{  S_\textbf{\emph{i}}(U_k) \}_{\textbf{\emph{i}} \in \mathcal{I}^{i,m}_\infty, k \in \mathcal{K}}
\]
is a cover of $\bigcup_{\textbf{\emph{i}} \in \mathcal{I}^{i,m}_\infty}  S_\textbf{\emph{i}}(F_A^i)$.  Let $\delta>0$ and choose $m \in \mathbb{N}$ sufficiently large to ensure that
\[
\sup_{\textbf{\emph{i}} \in \mathcal{I}^{i,m}_\infty, \, k \in \mathcal{K}} \text{diam}\big(  S_\textbf{\emph{i}}(U_k) \big) \leq \delta
\]
and thus
\[
\{  S_\textbf{\emph{i}}(U_k) \}_{\textbf{\emph{i}} \in \mathcal{I}^{i,m}_\infty, \,  k \in \mathcal{K}}
\]
is a countable open $\delta$-cover of $\bigcup_{\textbf{\emph{i}} \in \mathcal{I}^{i,m}_\infty}  S_\textbf{\emph{i}}(F_A^i)$.  It follows  that
\begin{eqnarray*}
\mathcal{H}^s_\delta(F_A^i)  &\leq&  \mathcal{H}^s_\delta \left(\bigcup_{\textbf{\emph{i}} \in \mathcal{I}^{i,m}_\infty}  S_\textbf{\emph{i}}(F_A^i)  \right) \ + \  \mathcal{H}^s_\delta \left(F_A^i \setminus \bigcup_{\textbf{\emph{i}} \in \mathcal{I}^{i,m}_\infty}  S_\textbf{\emph{i}}(F_A^i) \right) \\ \\
&\leq& \sum_{k \in \mathcal{K}} \sum_{\textbf{\emph{i}} \in \mathcal{I}^{i,m}_\infty} \text{diam}\big( S_\textbf{\emph{i}}(U_k) \big)^s  \qquad \qquad \text{by (\ref{222})} \\ \\
&\leq&  \sum_{k \in \mathcal{K}} \text{diam}( U_k )^s \sum_{\textbf{\emph{i}} \in \mathcal{I}^{i,m}_\infty} r_\textbf{\emph{i}}^s  \\ \\
&\leq& \Big(\mathcal{H}^s_\infty(F_A^i) \, + \, \varepsilon \Big) \, \left(\sum_{\textbf{\emph{i}} \in \mathcal{I}^{i}_\infty} r_{\tau(\textbf{\emph{i}})}^s\right)^m  \qquad \qquad \text{by (\ref{111}) and (\ref{abab})} \\ \\
&=& \mathcal{H}^s_\infty(F_A^i) \, + \, \varepsilon
\end{eqnarray*}
by property \emph{(vi)} from Lemma \ref{exhaust}.  Taking the limit as $\delta \to 0$ and noting that $\varepsilon>0$ was arbitrary, yields $\mathcal{H}^s(F_A^i)  \leq \mathcal{H}^s_\infty(F_A^i) $.  The reverse inequality is always satisfied.
\\ \\
The final part of Theorem \ref{main} follows by a simple trick. Let  $i \in \mathcal{I}$ and observe that
\[
F_A^i \ = \ S_i\left( \bigcup_{j \in \mathcal{I} : A_{i,j} = 1}  F_A^j \right)
\]
and so
\[
r_i^s \, \mathcal{H}_\infty^s  \left( \bigcup_{j \in \mathcal{I} : A_{i,j} = 1}  F_A^j \right) \  = \ \mathcal{H}_\infty^s (F_A^i)  \  = \  \mathcal{H}^s (F_A^i)  \  =  \  r_i^s \,  \mathcal{H}^s \left( \bigcup_{j \in \mathcal{I} : A_{i,j} = 1}  F_A^j \right)
\]
where the middle equality is due to the first part of the theorem.  Dividing by $r_i^s$ completes the proof.
\hfill \qed

\subsection{Proof of Theorem \ref{packing} and Theorem \ref{packing-gda}} \label{packingproof}

\emph{Proof of Theorem \ref{packing}}.
Let $F \subseteq \mathbb{R}^n$ be the self-similar attractor of the IFS $\{S_i\}_{i \in \mathcal{I}}$ and assume $F$ satisfies the strong separation condition.  This implies that we can find a bounded open set $\mathcal{O} \subseteq \mathbb{R}^n$ such that $F \subset \mathcal{O}$ and $\bigcup_{i \in \mathcal{I}}S_i(\mathcal{O}) \subseteq \mathcal{O}$ is a disjoint union.  Let
\[
\delta_0 \,  =  \, \frac {1}{2} \, \inf_{x \in F} \inf_{y \in\mathbb{R}^n \setminus\mathcal{O}} \  \lvert x-y \rvert
\]
which is strictly positive since $F$ is closed.
\\ \\
First assume that $\mathcal{P}^s_{\delta}(F)<\infty$ for every $\delta \in (0, \delta_0)$. Later we will see that $\mathcal{P}^s_{\delta}(F)=\infty$ is impossible for $\delta \in (0, \delta_0)$. Let $\varepsilon>0$, let $\delta \in (0, \delta_0)$ and let $\{ B_k\}_{k \in \mathcal{K}}$ be a countable collection of disjoint closed balls centered in $F$ with diameter less than or equal to $\delta$ which satisfies
\begin{equation} \label{333}
\sum_{k \in \mathcal{K}} \text{diam}( B_k )^s \  \geq \  \mathcal{P}^s_{\delta}(F) - \varepsilon.
\end{equation}
Since $B_k \subset \mathcal{O}$ for all $k \in \mathcal{K}$ and by the choice of $\mathcal{O}$, the collection
\[
\{ S_\textbf{\emph{i}}(B_k) \}_{\textbf{\emph{i}} \in \mathcal{I}^m, k \in \mathcal{K}}
\]
is a countable collection of disjoint closed balls centered in $F$.  Let $\eta \in (0, \delta)$ and choose $m \in \mathbb{N}$ so large so that
\[
\sup_{\textbf{\emph{i}} \in \mathcal{I}^m, k \in \mathcal{K}}\text{diam}\big(S_\textbf{\emph{i}}(B_k)\big) \, \leq \, \eta.
\]
It follows that
\begin{eqnarray*}
\mathcal{P}^s_{\eta}(F) &\geq&  \sum_{k \in \mathcal{K}} \sum_{\textbf{\emph{i}} \in \mathcal{I}^m} \text{diam}\big( S_\textbf{\emph{i}}(B_k)\big)^s \\ \\
&=&  \sum_{k \in \mathcal{K}} \text{diam}( B_k )^s \sum_{\textbf{\emph{i}} \in \mathcal{I}^m} r_\textbf{\emph{i}}^s  \\ \\
&\geq& \Big(\mathcal{P}^s_{\delta}(F) - \varepsilon \Big) \, \left(\sum_{i \in \mathcal{I}} r_i^s\right)^m \qquad \qquad \text{by (\ref{333})} \\ \\
&=& \mathcal{P}^s_{\delta}(F) - \varepsilon
\end{eqnarray*}
by the Hutchinson-Moran formula for (packing) dimension \cite{hutchinson}.  Taking the limit as $\eta \to 0$ and noting that $\varepsilon>0$ was arbitrary, yields $\mathcal{P}^s(F)=\mathcal{P}^s_{0}(F)   \geq  \mathcal{P}^s_{\delta}(F)   $.  The reverse inequality is always satisfied by (\ref{packineq}), which completes the proof.
\\ \\
Now assume that $\mathcal{P}^s_{\delta}(F)=\infty$ for some $\delta \in (0, \delta_0)$. Via a similar argument to the one above, this implies that $\mathcal{P}^s_{\eta}(F)>K$ for every $K>0$ and hence $\mathcal{P}^s_{0}(F)=\infty$ but this is a contradiction since every self-similar set has finite packing measure (and pre-measure) in the packing dimension, see \cite[Exercise 3.2]{techniques}. \hfill \qed
\\ \\
The reason this proof cannot be extended to the open set condition case is because in that case the number $\delta_0$ may be zero and iterations of packings may no longer be packings.  This is one of the reasons packings are sometimes more difficult to control than covers.  The proof of Theorem \ref{packing-gda} is similar and we just provide a sketch.  First we prove a simple lemma. We say $v \leq v'$ for vectors $v,v' \in \mathbb{R}^N$ if each entry in $v$ is less than or equal to the corresponding entry in $v'$. We say that $v$ is non-negative if $0\leq v$. Similar notations apply to matrices.
\begin{lma}
\label{matrixlem}
Let $A$ be a non-negative irreducible matrix of spectral radius $1$ and $x$ be a non-negative vector such that $A^m x\leq x$ for large enough $m$. Then $Ax=x$.
\end{lma}
\begin{proof}
Observe that $A^m$ is also an irreducible matrix with spectral radius $1$. Hence it follows from \cite[Theorem 1.3.28]{matrix} that $A^m x = x$ and therefore $Ax=x$ by the Perron-Frobenius theorem.
\end{proof}

\emph{Proof of Theorem \ref{packing-gda}}.
Let $A^s$ be the matrix with $(i,j)$th entry given by
\[
A^s_{i,j} \, = \, \sum_{e \in \mathcal{E}_{i,j}}r^s_e.
\]
Let $s$ be the unique value for which the spectral radius of the matrix $A^s$ is $1$. Let $u^{\intercal}=(\mathcal{P}^s(F_1),...,\mathcal{P}^s(F_N))$. If $\Gamma$ is strongly connected then $A^s$ is irreducible. If further the strong separation condition is satisfied then $0<\mathcal{P}^s(F_i)<\infty$ for every $i$ and $A^s u=u$ (see \cite[Corollary 3.5]{techniques}. Let $u_{\delta}^{\intercal}=(\mathcal{P}_{\delta}^s(F_1),...,\mathcal{P}_{\delta}^s(F_N))$. Since the strong separation condition is satisfied there exists a collection of open sets $\{\mathcal{O}_{i}\}_{i  \in \mathcal{V}}$ such that $F_i\subseteq \mathcal{O}_i$ and
\[
  \bigcup_{j=1}^N \bigcup_{e \in \mathcal{E}_{i,j}} S_e(\mathcal{O}_j) \, \subseteq \, \mathcal{O}_i
\]
is a disjoint union for every $i$. Let
\[
\delta_0 \,  = \, \frac {1}{2} \, \min_{i\in \mathcal{V}}\inf_{x \in F_i} \inf_{y \in\mathbb{R}^n \setminus \mathcal{O}_i} \  \lvert x-y \rvert.
\]
A similar argument to the proof of Theorem \ref{packing} shows that for large enough $m$ depending on $\eta$ we have that
\begin{equation} \label{Asdelta}
(A^s)^m u_{\delta} \,  \leq \,   u_{\eta} \,  \leq \,  u_{\delta}
\end{equation}
for $\delta \in (0,\delta_0)$ and $0<\eta<\delta$.   It follows by Lemma \ref{matrixlem} that equality holds in (\ref{Asdelta}). Hence $ u_{\delta}=u_{\eta}=u$ for $0<\eta<\delta<\delta_0$.
\hfill \qed

\section*{Acknowledgements}

\'A.F. was financially supported by an EPSRC doctoral training grant.  Most of this work took place whilst J.M.F. was a research fellow at the University of Warwick, where he was financially supported by the EPSRC grant EP/J013560/1.  \'A.F. visited J.M.F. at the University of Warwick to work on this project and both authors thank the department for its hospitality.  The authors thank Kenneth Falconer for helpful comments on the exposition of the paper.  Finally the authors thank Thomas Jordan and Mike Todd for providing some helpful references.

\end{document}